\documentclass[12pt,leqno]{amsart}
\usepackage[left=3.5cm,tmargin=2cm,bmargin=2.5cm,centering]{geometry}

\usepackage{hyperref}
\usepackage[utf8]{inputenc}
\usepackage[initials]{amsrefs}  
\usepackage{enumerate}
\usepackage{bm,amsmath,amssymb,amsthm}
\usepackage[all]{xy}
\usepackage{stmaryrd}
\usepackage{amsfonts}
\theoremstyle{plain}
\newtheorem{theorem}{Theorem}[section]

\newtheorem{corollary}[theorem]{Corollary}

\newtheorem{lemma}[theorem]{Lemma}
\newtheorem*{proposition*}{Proposition}
\newtheorem*{lemma*}{Lemma}
\newtheorem*{corollary*}{Corollary}

\theoremstyle{definition}
\newtheorem{definition}{Definition}[section]

\newtheorem{remark}[definition]{Remark}

\numberwithin{equation}{section}

\usepackage{dsfont}  

\renewcommand{\le}{\leqslant}
\renewcommand{\ge}{\geqslant}

\newcommand{\Mun}{\mathds{1}}
\newcommand{\Mdemi}{\frac{1}{2}}
\newcommand{\mdemi}{\tfrac{1}{2}}
\newcommand{\Mdede}[4]{
\begin{pmatrix}
#1&#2  \\
#3&#4 \\
\end{pmatrix}
}
\newcommand{\mdede}[4]{
\left(
\begin{smallmatrix}
#1&#2  \\
#3&#4 
\end{smallmatrix}
\right)
}
\DeclareMathOperator {\GL} {GL}

\newcommand{\BmC}{\mathbb{C}}

\newcommand{\BmN}{\mathbb{N}}

\newcommand{\BmR}{\mathbb{R}}

\DeclareMathOperator {\Msym} {sym}
\newcommand{\abs}[1]{\ensuremath{\left|#1\right|}}
\DeclareMathOperator {\MRe}  {\Re e}
\DeclareMathOperator {\MRes} {Res}

\newcommand{\Sprim}{\mathcal{S}^*_k(N)}
\newcommand{\SpM}{\mathcal{S}^*_\kappa(M,\chi)}
\newcommand{\tdemi}{\tfrac{1}{2}}

\begin{document}

\title[]{First moment of Rankin-Selberg central L-values and subconvexity in the level aspect}


\author[]{Roman Holowinsky}
\address{Department of Mathematics, The Ohio State University, 100 Math Tower, 231 West 18th Avenue, Columbus, OH 43210-1174}
\email{holowinsky.1@osu.edu}
\author[]{Nicolas Templier}   
\address{Department of Mathematics, Fine Hall, Washington Road, Princeton, NJ 08544-1000.}
\email{templier@math.princeton.edu}


\date{\today}

\begin{abstract}
Let $1\le N<M$ with $N$ and $M$ coprime and square-free. Through classical analytic methods we estimate the first moment of central $L$-values 
$
L(\tfrac{1}{2},f\times g)
$
where $f\in \Sprim$ runs over primitive holomorphic forms of level $N$ and trivial nebentypus and $g$ is a given form of level $M$. As a result, we recover the bound 
$
L(\tfrac{1}{2},f\times g) \ll_\varepsilon (N + \sqrt{M}) N^\varepsilon M^\varepsilon
$
when $g$ is dihedral.  The first moment method also applies to the special derivative $L'(\mdemi,f\times g)$ under the assumption that it is non-negative for all $f\in \Sprim$.
\end{abstract}

\maketitle


\tableofcontents


\section{Introduction}\label{sec:intro} 

In this paper, we investigate the subconvexity problem for certain Rankin-Selberg $L$-functions in the level aspect. A special feature, compared to other works in the subject, is that the method is relatively straightforward, at least in the prime level case of \S\ref{sec:pf}.   
This seems to suggest that special restrictions on the automorphic cusp form $\pi$ simplifies the subconvexity problem for $L(\mdemi,\pi)$.

The restrictions on $\pi$ are essentially of two kinds. The first assumption is that $L(\mdemi,\pi)$ be non-negative. This should be the same as requiring $\pi$ to be self-dual. For an unconditional result, we require further that $\pi$ be symplectic, so that the theorem of Lapid--Rallis~\cite{LR03} applies. The second assumption is that $\pi$ be a Rankin-Selberg convolution of two forms whose levels are coprime and of different size.  These restrictions are non-generic in the sense that we capture a small portion of the whole set of automorphic forms. They are nevertheless interesting, because such forms $\pi$ are easily constructed and often occur in applications to arithmetic problems.

The first occurrence of such a result is in the work of Michel--Ramakrishnan~\cite{MR07} and in the second-named author's PhD thesis~\cite{Temp:these}. The subconvexity bounds came for free, in some cases, as a direct result of an exact formula for the first moment of central $L$-values (see \cite{MR07}*{Corollary~2}). The formula is based on the Fourier expansion of the kernel occurring in the first step of the proof of the Gross formula. It was rather striking that subconvexity followed as an immediate corollary to an exact formula for the first moment and therefore Michel and Ramakrishnan raised the question as to whether a purely analytic proof would be possible.

The results of Michel--Ramakrishnan~\cite{MR07} have been greatly generalized by Feigon--Whitehouse~\cite{FW08} using the relative trace formula and Waldspurger formula. Notably the subconvexity bounds have been extended to number fields with the same quality of exponents. The approach in~\cite{Temp:these} is based on Zhang's notion of geometric pairing of CM cycles~\cite{Zhan01b} and the explicit computations in the original Gross--Zagier article~\cite{GZ}. The approach through Fourier expansions has been reinterpreted by Nelson~\cite{Nelson:stable} based on the computations by Goldfeld--Zhang of the kernels for the Rankin--Selberg convolution. Nelson's work is perhaps the closest to ours analytically, but focuses on stable averages rather than subconvexity. 

Our present approach is conceptually simpler than in the previous works.
We rely on the Rankin-Selberg convolution rather than explicit period formulas and use little input from spectral theory. The treatment of the first moment is indeed purely analytic, using well-known analytic tools: approximate functional equation, Petersson trace formula, Vorono\"i formula; it is therefore flexible enough to adapt to seemingly more complicated cases. We illustrate this aspect with a subconvex bound for the special derivative in the Gross-Zagier formula, in which case our result is new.

Let $M\ge 1$ be a square-free integer and let $g$ be a holomorphic cusp form of fixed weight $\kappa \ge 2$ on the congruence subgroup
\begin{equation}
\Gamma_1(M)=\{\mdede{a}{b}{c}{d},\ ad-bc=1,\ M|c,\ M|(a-1),\ M|(d-1)\}.
\end{equation}
  In the present context we can assume indifferently that $g$ is a newform because the $L$-functions associated to $g$ depend only on the automorphic representation it generates. Let $\chi$ be the Dirichlet character modulo $M$ which is the nebentypus of $g$. 

  Now let $N$ be a square-free integer coprime with $M$ and $k\ge 2$ be a fixed even integer. We let $S_k(N)$ be the vector space of holomorphic cusp forms of level $N$, weight $k$ and trivial nebentypus. It is equipped with a Petersson inner product and an action of the Hecke operators $T_n$. We denote by $\Sprim$ the set of primitive newforms and  $\omega_f=N^{1+o(1)}$ the associated spectral weights as in~\S\ref{s:petersson}.

Let $f\in \Sprim$ and denote by $\lambda_f$ and $\lambda_g$ the normalized Hecke eigenvalues of $f$ and $g$ as described in \S\ref{sec:autforms}. The Rankin-Selberg $L$-function $L(s,f\times g)$ of degree four
\begin{equation}
L(s,f\times g)=L^{(N)}(2s,\chi)\sum_{n\ge 1} \frac{\lambda_f(n)\lambda_g(n)}{n^s}
\end{equation}
admits an analytic continuation to all of $\BmC$ and a functional equation
\begin{equation}
  \Lambda(s,f\times g)= (\frac{NM}{4\pi})^s L_\infty(s,f\times g) L(s,f\times g) = \epsilon(f\times g) \Lambda(1-s,f\times \overline{g}).
\end{equation}
Here $L_\infty(s,f\times g)$ is a product of two complex $\Gamma$ factors,  see~\S\ref{sec:pre:RS} for details. The (arithmetic) conductor equals $\mathcal{Q}:=(NM)^2$. It happens often that the root number $\epsilon(f\times g)$ depends only on $N$ and $g$, such as in the present case where $M$ and $N$ are coprime. 
We now give a statement of our main theorem.
\begin{theorem}\label{th:intro} Let $1\le N<M$ with $N$ and $M$ coprime and square-free. For $g\in \SpM$ we have
\begin{equation}
  \sum_{f\in \Sprim} \omega_f^{-1} L(\tfrac{1}{2},f\times g)\ll_\varepsilon \left(  1 + \frac{\sqrt{M}}{N} \right) N^\varepsilon M^\varepsilon.
\end{equation}
The same bound holds with $L^{(j)}(\mdemi+it, f \times g)$ for any fixed integer $j\ge 0$ and $t\in \BmR$.
\end{theorem}

If one is willing to apply Theorem~\ref{th:intro} to the subconvexity problem, it is necessary to have an extra assumption. We would want each $L$-value $L(\mdemi,f\times g)$ to be non-negative. In the next subsection we proceed to discuss the significance of this assumption and the instances where it is known to hold.  When one does have non-negativity, we obtain the following bound for an individual $L$-function.
\begin{corollary}\label{co:intro} Let $g\in \SpM$ and suppose $L(\mdemi,f\times g)\ge 0$ for all $f\in \Sprim$. Then:
  \begin{equation}\label{boundL12}
L(\tfrac{1}{2},f\times g) \ll_\varepsilon \left(  N + \sqrt{M} \right) N^\varepsilon M^\varepsilon
\end{equation}
In particular if $N=M^{\frac12+o(1)}$ then
\begin{equation}
  L(\tfrac{1}{2},f\times g) \ll \mathcal{Q}^{\frac16+o(1)}.
\end{equation}
The same bounds hold for the first derivative under the same assumption that $L'(\mdemi,f\times g)\ge 0$ for all $f\in \Sprim$.
\end{corollary}
Note that~\eqref{boundL12} is better than the convexity bound when $M^\delta<N<M^{1-\delta}$ for any $\delta>0$ and matches Corollary~2 in \cite{MR07} and Theorem~1.4 in \cite{FW08}.
The above Theorem and Corollary are made possible by the following main estimate.
\begin{lemma}\label{lem:VoronoiApp}
  Let $Z\ge 1$. Let $k$ and $\kappa$ be fixed positive integers.  Let $q, D, M$ be positive integers with $M$ square-free.  Let $g\in \SpM$ be a newform of weight $\kappa$, level $M$ and nebentypus~$\chi$. Then the sum 
$$
 \sum_{n}^{}
  \frac{\lambda_g(n)}{\sqrt{n}}
S\Big(n D,1;q\Big)
  J_{k-1}\left( \frac{4\pi \sqrt{nD}}{q} \right)
h\left(\frac{n}{Z}\right)
$$
where $h$ is a smooth function, compactly supported in $[1/2,5/2]$ with bounded derivatives is
$$
 \ll_\varepsilon \left(\sqrt{\frac{Z}{M_2}}+q_2 \big(1+P\big)\sqrt{\frac{M_2}{Z}}\right) \frac{P^2}{(1+P)^3} (qDMZ)^{\varepsilon}
$$
where $P=\dfrac{\sqrt{ZD}}{q}$, $q_2=\dfrac{q}{(D,q)}$ and $M_2=\dfrac{M}{(q_2,M)}$.
\end{lemma}

\begin{remark}
One should compare this bound with the ``trivial'' bound 
$
\sqrt{Zq} P(1+P)^{-3/2} (qDMZ)^{\varepsilon}
$
obtained by an application of the Weil bound for individual Kloosterman sums along with \eqref{Joss} and \eqref{JWbounds}.  In particular, consider the case of the transition range for the Bessel function, i.e. $P\sim 1$. Furthermore, one may slightly relax the conditions for bounds on the smooth function $h$ and still obtain similar results.
\end{remark}

\subsection{Non-negativity of central values.}\label{sec:intro:positivity} 
 We discuss in this subsection the above question of non-negativity of $L(\mdemi,f\times g)$ and $L'(\mdemi,f\times g)$. We review what is known unconditionally and what is expected. Non-negativity of central values is a deep fact and there are several works that rely on using this property.  Indeed, non-negativity allows for the study of moments of odd order in application to subconvexity. See notably Ivic~\cite{Ivic:nonnegative}, Conrey--Iwaniec~\cite{CI00}, Li~\cite{Li:subconvex} and Blomer~\cite{Blom:GL3}.   We begin by recalling what is known in the general case and then draw the consequences in our setting.  This is related to the classification of automorphic forms and the Gross--Zagier formula. 

For an irreducible $L$-function $L(s,\pi)$ to have real Dirichlet coefficients we assume that $\pi$ is self-dual. Then $L(s,\pi)>0$ for all $s\ge 1$. Assuming GRH it would follow that $L(\mdemi,\pi)\ge 0$ and that if $L(\mdemi,\pi)=0$ then $L'(\mdemi,\pi)\ge 0$.

We now recall in which cases these inequalities are known unconditionally. According to the Arthur classification, the cuspidal self-dual representation $\pi$ is either symplectic or orthogonal. Namely exactly one of the $L$-functions $L(s,\pi, \wedge^2)$ (adjoint) or $L(s,\pi,\Msym^2)$ (symmetric square) has a (simple) pole at $s=1$. We say that $\pi$ is symplectic (resp. orthogonal) when the adjoint (resp. symmetric square) $L$-function has a pole. A general theorem of Lapid-Rallis~\cite{LR03} says that if $\pi$ is symplectic then $L(\mdemi,\pi)\ge 0$. If $\pi$ is orthogonal it is not known unconditionally that $L(\mdemi,\pi)\ge 0$, though we expect moreover that $L(\mdemi,\pi)>0$.
For example when $\pi$ is a $\GL(1)$ quadratic Dirichlet character this is an open question related to the effective class number problem~\cite{Iwaniec:conversations}*{\S4}.

Let $\pi=f\times g$ be the Rankin--Selberg convolution. The existence of $\pi$ as an automorphic representation on $GL(4)$ is established by Ramakrishnan~\cite{Rama00}. It remains to investigate under which conditions on the forms $f$ and $g$ the representation $\pi=f\times g$ can be self-dual and under which conditions it can be symplectic. 

We consider only the case when $f$ and $g$ are both self-dual which assures that $\pi$ is self-dual. It is not difficult to determine whether a $GL(2)$ form is self-dual because the contragredient of a $GL(2)$ form is its twist by the inverse of the central character. Thus it is necessary that the central characters of $f$ and $g$ are quadratic (i.e. real and valued in $\{\pm 1\}$). See also~\cite{book:Iwan:topic}*{Chap. 6} where, in the classical language, it is shown that if the nebentypus is a quadratic character then the $GL(2)$ form is an eigenvector of the Fricke involution. Note that the central character of $\pi$ should be trivial because it is the square of the product of the central characters of $f$ and $g$. 

For $\pi$ to be symplectic it is necessary and sufficient that one of the forms $f$ and $g$ be orthogonal and the other be symplectic.  A $GL(2)$ form is symplectic if and only if its central character is trivial. It seems advantageous to average over a family of symplectic forms. This is why in Theorem~\ref{th:intro} we average over the family $f\in \Sprim$, while the other form $g$ will be assumed to be orthogonal. 

A $GL(2)$ form $g$ is orthogonal if
and only if it is dihedral. The central character is always non-trivial since
it equals the character of the quadratic extension it is associated with.  
The average over the dihedral forms $g$ also can be considered, see~\cites{LMY12,Temp:shifted} and the references there.

Let $g\in \SpM$. We have arrived at the following conclusion concerning the non-negativity assumption for the special value.
\begin{itemize}
  \item If $g$ is self-dual (that is if $\chi$ is quadratic) then we expect $L(\mdemi,f\times g)\ge 0$ for all $f\in \Sprim$;
  \item this is known unconditionally if $g$ is dihedral, which is the case treated in~\cites{MR07,FW08}. 
\end{itemize}

Next we need to analyse the situation at a finer level in order to take into account the sign of the root number. Assume that $\pi$ is self-dual; the root number $\epsilon(f\times g)=\epsilon(\mdemi,\pi)$ is $\pm 1$. In fact for $f\in \Sprim$ and $g\in \SpM$ self-dual we will see in~\eqref{efg} that
\begin{equation*}
  \epsilon(f\times g) = 
  \begin{cases}
	\chi(-N), & \text{if $\kappa\ge k$,}\\
	\chi(N),  & \text{if $k\ge \kappa$.}
  \end{cases}
\end{equation*}
If $\epsilon(f\times g)=-1$ we have $L(\mdemi,f\times g)= 0$ in which case it is interesting to focus on the derivative $L'(\tfrac12,f\times g)$. Its non-negativity is not yet known in general (e.g. if one of the forms in the convolution was taken to be Maass, a case not considered here). However, non-negativity is known in an important special case which is the celebrated Gross-Zagier formula~\cite{GZ}. In our context of Theorem~\ref{th:intro} this is known by the recent work of~Yuan--Zhang--Zhang~\cite{YZZ:GZ} if $f$ and $g$ have weight $k=\kappa=2$. 

We should also mention the Waldspurger theorem~\cite{Wald85} which gives a period formula for $L(\mdemi,f\times g)$ when $f$ is symplectic and $g$ is dihedral and also implies non-negativity. In the present paper we do not use the Waldspurger period formula in the proof of Theorem~\ref{th:intro} which is an important difference with~\cites{MR07,FW08}. Non-negativity is a robust property which can be established independently of period formulas as shown by the results of Lapid--Rallis~\cites{LR03,Lapid:RS}.

\subsection{Detailed outline of the first moment method}\label{s:sketch} We illustrate the main ideas in this subsection by giving a detailed outline of the first moment method under simplifying assumptions.

First assume that we have two distinct prime levels $N$ and $M$ with $1<N<M$. We fix a newform $g\in \SpM$ and shall average $L(\frac{1}{2}, f\times g)$ over the collection of newforms $f\in \mathcal{S}^{\ast}_{k}(N)$ with trivial nebentypus. We further assume that the space of cusp forms  is equal to the space of newforms as happens to be the case for example when the weight $k\leqslant 10$ or $k=14$ and the level $N$ is prime since for such $k$ there are no oldforms of full level. After a standard approximate functional equation argument, we see that the problem reduces to obtaining estimates for sums of the form
\begin{equation}
\sum_{f\in \Sprim} \omega_f^{-1} \sum_{n}\frac{\lambda_f(n)\lambda_g(n)}{\sqrt{n}} V\left(\frac{n}{NM}\right)\label{sketchstart}
\end{equation}
with $V$ as in \S\ref{sec:pre:app}. Assuming non-negativity for each $f$, bounding the above by $(M/N)^{1/2}$ would produce the convexity bound for any individual $L$-function.  We note that such a bound is weaker than Lindel\"{o}f on average over $f$ if $M^\delta<N<M^{1-\delta}$ for any $\delta>0$.

Summing over $f$ via Petersson's trace formula, we arrive at 
$$
\sum_{n}\frac{\lambda_g(n)}{\sqrt{n}}V\Bigl(
\frac{n}{NM}
\Bigr)\left\{\delta(n,1)+2\pi i^{-k}\sum_{\substack{c>0\\c\equiv 0(N)}}\frac{1}{c}S(n,1;c) J_{k-1}\left(\frac{4\pi\sqrt{n}}{c}\right)\right\}.
$$
 The ``diagonal'' delta term contributes $V(1/NM)$ which is of size $1$.  This prevents one from proving better than Lindel\"{o}f on average over the family.  To treat the ``off-diagonal'' sums, we proceed in a manner motivated by the earlier works of Kowalski--Michel--Vanderkam~\cite{KMV02}, Harcos--Michel~\cite{HM06} and Michel~\cite{Mich04} and also seen recently in \cite{HoMu}.
 
We start by breaking the $n$-sum into dyadic segments, of lengths $Z$ say, through a smooth partition of unity with some nice compactly supported test function $h$.  For the purpose of this outline, we focus on the case of $Z=NM$. The Weil bound for individual Kloosterman sums and standard bounds for the Bessel functions (see~\S\ref{sec:pre:bessel}) allow us to initially truncate the $c$-sum to length $(NM)^{A}$ with the tail bounded by $(NM)^{-B}$ for some positive $A$ and $B$.  Furthermore, since $M$ is prime, we can restrict to those $c$ which are coprime with $M$ by using the same bounds. 

We now proceed with our analysis of these off-diagonal terms by reducing to ``shifted sums''.  To achieve this, we change the order of summation
\begin{equation}
\sum_{\substack{c\leqslant (NM)^A \\ c\equiv 0(N)\\ (c,M)=1}}\frac{1}{c} \sum_n \frac{\lambda_g(n)}{\sqrt{n}}S(1,n;c) J_{k-1}\left(\frac{4 \pi \sqrt{n}}{c}\right) h\left(\frac{n}{NM}\right). \label{sketchOpenK}
\end{equation}
and consider the inner $n$-sum for each fixed $c$.   What follows are the main ideas behind Lemma~\ref{lem:VoronoiApp}; the details appear in \S\ref{sec:pflem}.  

Opening the Kloosterman sum, an application of Vorono\"{i} summation in $n$ changes the inner sum in \eqref{sketchOpenK}, up to some bounded constant factor, to
$$
\frac{\sqrt{N}}{c}\sum_n \lambda_{g^\ast}(n)\sideset{}{^\ast}\sum_{\alpha(c)} e\left(\frac{\alpha(1-n\overline{M})}{c}\right)I_c(n) 
$$
for some other newform $g^\ast$ of weight $\kappa$ and level $M$ with 
$$
I_c(n)= \int_0^{\infty}\frac{h(\xi)}{\sqrt{\xi}}J_{k-1}\left(\frac{4\pi \sqrt{\xi NM}}{c}\right)J_{\kappa-1}\left(\frac{4\pi \sqrt{n\xi N}}{c}\right) d\xi.
$$
Analysis of the integral $I_c(n)$ restricts the $n$-sum to roughly those $n$ satisfying $|\sqrt{M}-\sqrt{n}| \ll \dfrac{c}{\sqrt{N}}$. Furthermore, we have the bound $I_c(n) \ll \dfrac{P^2}{(1+P)^3}$ with $P=\dfrac{\sqrt{NM}}{c}$ by Lemma~\ref{lem:Icn} (whose proof is given in the \S\ref{sec:pre}).

The arithmetic advantage of Vorono\"{i} summation is that one now has Ramanujan sums instead of Kloosterman sums for each modulus $c$. Such an idea is well-known and was already seen in a work of Goldfeld \cite{G}. We write the Ramanujan sums as
$$
\sideset{}{^\ast}\sum_{\alpha(c)} e\left(\frac{\alpha(M-n)}{c}\right)=\sum_{\delta \nu = c} \mu(\delta) \sum_{\alpha(\nu)}e\left(\frac{\alpha(M-n)}{\nu}\right)
$$
and the inner sums over $\alpha$ will detect the congruence condition $n\equiv M (\nu)$ with a loss of $\nu=c/\delta$. Therefore, we have reduced \eqref{sketchOpenK} to bounding
\begin{equation}
N^{3/2} M\sum_{\substack{c\leqslant (NM)^A\\ c\equiv 0(N)\\ (c,M)=1}}\frac{1}{c^3(1+\frac{\sqrt{NM}}{c})^{3}}\sum_{\delta \nu = c}\frac{1}{\delta}\sum_{\substack{n\equiv M(\nu)\\|\sqrt{M}-\sqrt{n}| \ll \tfrac{c}{\sqrt{N}}}} |\lambda_g^\ast(n)|.\label{sketchShift}
\end{equation}
The Ramanujan-Petersson bound for Hecke eigenvalues then allows one to conclude that
$$
\sum_{f\in \Sprim} \omega_f^{-1} \sum_{n}\frac{\lambda_f(n)\lambda_g(n)}{\sqrt{n}} h\left(\frac{n}{NM}\right) \ll \frac{\sqrt{M}}{N}(NM)^{\varepsilon}.
$$
Note that, unlike in the case of the second moment \cite{HoMu}, the ``zero shift'' $n=M$ did not need to be treated separately to obtain the above bound, i.e. one does not need to use the additional fact (see \eqref{SpecialHecke}) that $|\lambda_g^\ast(M)|=M^{-1/2}$. Combining this with the contributions of the other dyadic segment in $n$ gives Theorem~\ref{th:intro} and finally Corollary~\ref{co:intro}
\begin{equation}
L( \tfrac{1}{2}, f\times g) \ll_\varepsilon \sqrt{NM}\left(\frac{\sqrt{N}}{\sqrt{M}}+\frac{1}{\sqrt{N}}\right)(NM)^{\varepsilon} \label{moral}
\end{equation}
for any individual $L$-function in our family. 

We have written the subconvexity bound in the above form in order to comment on the significance of each term appearing on the right hand side of \eqref{moral}. The first term comes from the trivial diagonal in Petersson's trace formula and represents Lindel\"{o}f on average. The ratio $(N/M)^{1/2}$ provides a natural upper boundary for the size of $N$ relative to $M$ and shows that if $N$ and $M$ are of the same size, then one loses the advantage of having analytically distinguishable choices in which family to average over.  The second term comes from analysis of the off-diagonal in Petersson's trace formula.  The ratio $(1/N)^{1/2}$ provides the natural lower boundary and shows that $N$ must have some significant size relative to $M$ in order for our first moment average to be non-trivial.

\section{Preliminaries}\label{sec:pre}

\subsection{Bessel functions}\label{sec:pre:bessel}
We record here some standard facts about the $J$-Bessel functions as can be seen in~\cite{book:WW} as well as several estimates for integrals involving Bessel function which will be required for our application.  One may write the $J$-Bessel functions as 
\begin{equation}
J_k(x) = e^{ix} W_k(x)+e^{-ix} \overline{W}_k(x) \label{Joss}
\end{equation}
where 
\begin{equation}
W_k(x)=\frac{e^{i(\frac{\pi}{2}k-\frac{\pi}{4})}}{\Gamma(k+\frac{1}{2})}\sqrt{\frac{2}{\pi x}} \int_0^{\infty} e^{-y}(y(1+\frac{iy}{2x}))^{k-\frac{1}{2}}dy \label{JW}
\end{equation}
which, when $k$ is a positive integer, one has that
\begin{equation}
x^j W_k^{(j)}(x) \ll \frac{x}{(1+x)^{3/2}}. \label{JWbounds}
\end{equation}
Using the above facts leads us to the following results.

\begin{lemma}\label{lem:Icn}
  Let $k, \kappa\ge 2$ be fixed integers and let $a,b>0$. Define
  $$
  I(a,b):= \int_{0}^\infty \frac{h(\xi)}{\sqrt{\xi}} J_{k-1}(4 \pi a\sqrt{\xi}) J_{\kappa-1}(4 \pi b\sqrt{\xi}) d\xi,
  $$
where $h$ is a smooth function compactly supported on $\left[\frac{1}{2},\frac{5}{2}\right]$ with bounded derivatives. We have 
\begin{equation} 
  I(a,b)\ll_{j,h}  \frac{a}{(1+a)^{3/2}} \frac{b}{(1+b)^{3/2}}\abs{a- b}^{-j} 
\end{equation}
for any $j\ge 0$.
\end{lemma}

\begin{proof}
A change of variables, $\xi=w^2$, gives
$$
I(a,b)=2\int_0^{\infty} h(w^2)\;  J_{k-1}\left(4\pi a w\right)J_{\kappa-1}\left(4 \pi b w\right)dw.
$$
Therefore, we see from \eqref{Joss} that $I(a,b)$ may be written as the sum of four similar terms, one of them being
$$
\int_0^{\infty} e\left(2w(a-b)\right)h(w^2) \;W_{k-1}\left(4\pi a w\right)\overline{W}_{\kappa-1}\left(4\pi b w\right)dw.
$$
Repeated integration by parts and an application of \eqref{JWbounds} gives the desired result.
\end{proof}

\subsection{Automorphic forms} \label{sec:autforms} We have two integers $M\ge 1$ and $\kappa\ge 2$ and $\chi$ a Dirichlet character of modulus $M$. Recall that we denote by $\mathcal{S}_\kappa(M,\chi)$ the vector space of weight $\kappa$ holomorphic cusp forms with level $M$ and nebentypus $\chi$. 
We have the Fourier expansion
\begin{equation}
g(z)=\sum_{n\ge 1}
\psi_g(m)m^{(\kappa-1)/2}
e(mz).
\end{equation}
The space $\mathcal{S}_\kappa(M,\chi)$ is equipped with the Petersson inner product
\begin{equation} 
\langle g_1,g_2 \rangle_M=
\int_{X_0(M)}^{}
g_1(z)\overline{g_2(z)} y^{\kappa-2} dxdy.
\end{equation}

We recall the Hecke operators $T_n$ with $(n,M)=1$. The adjoint of $T_n$ with respect to the Petersson inner product is $T_n^*=\overline{\chi(n)}T_n$, hence $T_n$ is normal. There is an orthogonal basis of $\mathcal{S}_\kappa(M,\chi)$ consisting of eigenvectors of all the Hecke operators $T_n$ with $(n,M)=1$.

The subspace of newforms of $\mathcal{S}_\kappa(M,\chi)$ is the orthogonal complement of the subspace generated by old forms of type $g(dz)$ with $g$ of level strictly dividing $M$. The set of primitive forms $\mathcal{S}_\kappa^*(M,\chi)$ is an orthogonal basis of the subspace of newforms. A primitive form $g$ is a newform which is an eigenfunction of all $T_n$ with $(n,M)=1$ and such that $\psi_g(1)=1$.

A primitive form is actually an eigenfunction of all Hecke operators, and $\lambda_g(n)=\psi_g(n)$ is the normalized eigenvalue for all $n\ge 1$. We have the Hecke relation
\begin{equation}
\lambda_g(m)\lambda_g(n)=\sum_{\substack{d\mid (m,n)\\(d,M)=1}}
\chi(d)\lambda_g(\frac{mn}{d^2})
,\quad 
m,n\ge 1.
\end{equation}
We also record here that for $g$ primitive with trivial character one has that 
\begin{equation}
\lambda_g(m^2)=\lambda_g(m)^2=m^{-1} \label{SpecialHecke}
\end{equation}
for $m | M$ (see~\cite{ILS00}*{(2.24)}).

When the nebentypus is trivial we remove it from the notation. Our primitive form $f$ of level $N$ is an element of $\Sprim$.

\subsection{Rankin-Selberg $L$-functions}\label{sec:pre:RS}
Let $f\in \Sprim$ and $g\in \SpM$ with $N$ and $M$ squarefree and $(N,M)=1$. We recall the Rankin-Selberg $L$-function
\begin{equation}\label{def:Lfg}
L(s,f\times g)=L^{(N)}(2s,\chi) \sum_{n\ge 1}^{} \frac{\lambda_f(n)\lambda_g(n)}{n^s}.
\end{equation}
It admits an analytic continuation to all of $\BmC$ and a functional equation of the form
\begin{equation}\label{def:Lambdafg}
\Lambda(s,f\times g)= (\frac{NM}{4\pi^2})^s L_\infty(s,f\times g) L(s,f\times g) = \epsilon(f\times g) \Lambda(1-s,f\times g).
\end{equation}

The product of Gamma factors reads 
\begin{equation}\label{Lfg}
L_\infty(s,f\times g)=
\Gamma(s+\frac{\abs{k-\kappa}}{2})\Gamma(s+\frac{k+\kappa}{2}-1).
\end{equation}

According to~\cite{KMV02} the epsilon factor equals
\begin{equation}\label{efg}
\epsilon(f\times g)=
\begin{cases}
\chi(-N)\eta_g(M)^2 & \text{if $\kappa\ge k$,} \\
\chi(N) \eta_g(M)^2 & \text{if $k\ge \kappa$.}
\end{cases}
\end{equation}
Here $\eta_g(M)$ is the pseudo-eigenvalue of $g$ for the Atkin--Lehner operator $W_M$. The formula is consistent if $k=\kappa$ because $k$ is even which implies that $\chi(-1)=1$ in such case. References for the Fricke involution and related constructions include Chapter 6 of~\cite{book:Iwan:topic}, Li~\cite{Li79} and the Appendix A.1 in~\cite{KMV02}.

In particular $\epsilon(f\times g)$ depends only on $N$ and $g$. This enables us to perform the average over $f\in\Sprim$ in Theorem~\ref{th:intro}.
If $g$ is self-dual (equivalently if $\chi$ is real), then $\eta_g(M)=\pm 1$ is the product of the root numbers of $g$ at the primes dividing $M$; thus $\epsilon(f\times g)=\pm 1$ depends only on $N$, $\chi$ and $k,\kappa$ as noted in the introduction. 

The formulas for the gamma factor~\eqref{Lfg} and the epsilon factor~\eqref{efg} can be found in~\cite{KMV02}*{\S4} who quote~\cite{Li79}*{Th. 2.2}. However it is more satisfactory to verify the functional equation with the framework of automorphic representations, as may be found e.g. in Jacquet~\cite{book:JLII}. For the sake of completeness we provide some details on how to derive~\eqref{Lfg} and~\eqref{efg} in this way; we fix a (standard) additive character $\psi$ and proceed place by place. 

\subsubsection{Real place} The component of $f$ (resp. $g$) at infinity is the discrete series representation of weight $k$ (resp. $\kappa$). The central  character is $\Mun=sgn^k$ (resp. $sgn^\kappa$). Let $W_\BmR=\BmC^\times \cup j\BmC^\times$ be the Weil group.

Under the local Langlands correspondence the discrete series representation of weight $k$ corresponds to the two dimensional representation of $W_\BmR$ given by
\begin{equation} 
  z\mapsto \Mdede{(\frac{z}{\bar z})^{\frac{k-1}{2}}}
  {0}{0}{(\frac{z}{\bar z})^{\frac{1-k}{2}}}
 ,\quad
 j\mapsto \Mdede{0}{(-1)^{k}}{1}{0}.
\end{equation}
The gamma factor is $\Gamma_\BmC(s+\frac{k-1}{2})$ and the epsilon factor is $i^k$.

A small computation shows that the tensor product of the representation of weight $k$ and the representation of weight $\kappa$ decomposes as the direct sum of two representations of weight $k+\kappa-1$ and $\max(k-\kappa,\kappa-k)+1$, respectively. This implies the formula~\eqref{Lfg} for the gamma factor, while the epsilon factor at infinity is
\begin{equation} \label{epsiloninfty}
\epsilon_\infty(f\times g,\psi) = (-1)^{\max(k,\kappa)}
=
\begin{cases}
\chi(-1),& \text{if $\kappa \ge k$,}\\
1,& \text{if $k \ge \kappa$.}
\end{cases}
\end{equation}

\subsubsection{Primes dividing $N$} Let $p\mid N$. The component of $f$ (resp. $g$) at $p$ is the Steinberg representation (resp. an unramified principal series representation with central character $\chi_p$). Using standard formulas for the epsilon factors (tensor product with an unramified representation~\cite{cong:auto77:tate}*{(3.4.6)}), we obtain
\begin{equation} \label{epsilonN}
\epsilon_p(f\times g,\psi) = \chi_p(p) \epsilon_p(f)^2 = \chi(p).
\end{equation}

\subsubsection{Primes dividing $M$} Let $p\mid M$. The component of $f$ (resp. $g$) at $p$ is an unramified principal series representation with trivial central character (resp. a ramified principal series representation with central character $\chi_p$). Using standard formulas for the epsilon factors 
\begin{equation} 
\label{epsilonM}
\epsilon_p(f\times g,\psi)=  \epsilon_p(g,\psi)^2 = \eta_g(p)^2.
\end{equation}
Here we used the fact that $\epsilon_p(g,\psi)=\eta_g(p)$ (the pseudo-eigenvalue at a prime $p$ is the same as the local root number).

The epsilon factor at unramified places is equal to $1$. Multiplying the identities~\eqref{epsiloninfty},~\eqref{epsilonN} and~\eqref{epsilonM} we derive~\eqref{efg}.

\subsection{Approximate functional equation}\label{sec:pre:app} The method is standard to express or approximate values of $L$-functions inside the critical strip (and actually goes back to Riemann); we shall briefly set it up for the Rankin--Selberg $L$-functions, see~\cite{Harc02,Mich04} and~\cite{book:IK04}*{\S5.2} for details.

We first treat the central value $L(\mdemi,f\times g)$ and assume that $\chi$ is non-trivial. We fix a meromorphic function $G$ on $\BmC$ which satisfies the following
\begin{enumerate}[(i)]
\item $G$ is odd, $G(s)=-G(-s)$;
\item $G$ is holomorphic except at $s=0$ where it has a simple pole with residue $\MRes_{s=0} G(s)=1$;
\item $G$ is of moderate growth (polynomial) on vertical lines.
\end{enumerate}
Then we construct the smooth function
\begin{equation}\label{def:V}
V(y) = \int_{\MRe s=2}
y^{-s}\widehat V(s)
\frac{ds}{2i\pi},\quad y\in (0,\infty)
\end{equation}
where
\begin{equation} \label{def:hatV}
\widehat V(s)=
\frac{L_\infty(\mdemi+s,f\times g)}{L_\infty(\mdemi,f\times g)} L^{(N)}(1+2s,\chi) G(s).
\end{equation}
The approximate functional equation method shows that the special value $L(\mdemi,f\times g)$ is given by 
\begin{equation}\label{appfneq}
  \sum_n
  \frac{\lambda_f(n)\lambda_g(n)}{\sqrt{n}}
V\Bigl(
\frac{n}{NM}
\Bigr)
+
\epsilon(f\times g)
\sum_n
\frac{\overline{\lambda_f(n)}\overline{\lambda_g(n)}}{\sqrt{n}}
\widetilde{V}\Bigl(
\frac{n}{NM}
\Bigr)
\end{equation}
We have the following uniform estimates for the functions $V$ and $\widetilde V$. This follows by shifting the contour of integration in~\eqref{def:V}.
\begin{lemma}\label{lem:V} For every non-negative integer $\alpha\in \BmN$,
\begin{equation} 
V^{(\alpha)}(y)=
\begin{cases}
  L^{(N)}(1,\chi) \delta_{0,\alpha} + O_\alpha(y^{1/2-\alpha}), &\text {for $0<y\le 1$,}\\
O_{A,\alpha}(y^{-A}), &\text{for $1\le y$ and $A>0$.}
\end{cases}
\end{equation}
\end{lemma}

In Theorem~\ref{th:intro} we are also concerned with higher derivatives and general critical values. The approximation follows in the same manner (see also~\cite{KMV-nonvanishing}). The functions $V$ and $\widehat V$ are similarly defined by the same equations~\eqref{def:V} and~\eqref{def:hatV}.
We now let the meromorphic function $G$ on $\BmC$ be such that
\begin{enumerate}[(i)]
\item $G$ is holomorphic except at $s=it$, where we have a pole of order $j+1$ and
\begin{equation*} 
  \mathrm{Res}_{s=it} \widehat V(s) G(s) \sum_{n\ge 1} \frac{\lambda_f(n)\lambda_g(n)}{n^s}= L^{(j)}(\tfrac12+it,f\times g);
\end{equation*}
\item $G$ is of moderate growth (polynomial) on vertical lines.
\end{enumerate}
It may be verified that such a function always exists and may be chosen independently of $f\in \Sprim$ (indeed it only depends on $f$ through the gamma factors $L_\infty(s,f\times g)$ which are given in terms of $k$ and $\kappa$).

The approximate functional equation method shows that $L^{(j)}(\mdemi+it,f\times g)$ is again given by the same expression~\eqref{appfneq}. Lemma~\ref{lem:V} holds true as well, except that the term $\delta_{0,\alpha}$ has to be replaced by the $\alpha$-derivative of some polynomial in $\log y$ of degree at most $j+1$.



\subsection{Averaging over a family of forms}\label{s:petersson}
Let $k\geqslant 2$ be an integer.  For any $c,m,n \in \mathbb{N}$ let $S(m,n;c)$ denote the Kloosterman sum
$$
S(m,n;c)=\sideset{}{^\ast}\sum_{\alpha(c)} e\left(\frac{m \alpha+n\overline{\alpha}}{c}\right).
$$ 
Let $N\geqslant 1$ be an integer and let $\mathcal{B}_k(N)$ be any Hecke eigenbasis for $\mathcal{S}_k(N)$. Let $\mathcal{S}^{\ast}_k(N)$ denote the collection of newforms in $\mathcal{B}_k(N)$.  For any $m,n\geqslant 1$, set
$$
\Delta_{k,N}(m,n):=\sum_{f\in \mathcal{B}_k(N)} \omega_f^{-1}\psi_f(m)\overline{\psi_f(n)}
$$
and
$$
\Delta^\ast_{k,N}(m,n):=\sum_{f\in \mathcal{S}^{\ast}_k(N)} \omega_f^{-1}\lambda_f(m)\lambda_f(n)
$$
where the spectral weights $\omega_f$ are given by
$$
\omega_f:= \frac{(4\pi)^{k-1}}{\Gamma(k-1)}\langle f, f \rangle_N.
$$
Note that the inner-product is taken at the same level $N$ on which we are averaging our family of forms. This convention differs slightly from~\cite{ILS00} in which the inner-product is always taken at the largest ambient level.  

We have the following standard tool for averaging Fourier coefficients over a Hecke eigenbasis.
\begin{lemma}[Petersson trace formula] We have
$$
\Delta_{k,N}(m,n)= \delta(m,n)+2\pi i^{-k} \sum_{\substack{c>0\\c\equiv 0 (N)}}\frac{1}{c} S(m,n;c) J_{k-1}\left(\frac{4\pi \sqrt{mn}}{c}\right)
$$
where $\delta(m,n)=1$ if $m=n$ and $\delta(m,n)=0$ otherwise.
\end{lemma}

For the purpose of our application, we wish to write down a summation formula when the average is restricted to the family of newforms $\Sprim$.  Let
\begin{equation}
v(N)=\left[ \Gamma_0(1) : \Gamma_0(N)\right] = N \prod_{p|N} (1+p^{-1}). \label{vN}
\end{equation}
One has the following renormalized version of a result of Iwaniec, Luo and Sarnak~\cite{ILS00}*{Proposition~2.8}. Under the assumptions that $N$ be square-free, $(m,N)=1$ and $(n,N^2)|N$,
\begin{equation}\label{ILS}
\Delta^\ast_{k,N}(m,n)= \sum_{LR=N}\frac{\mu(L)}{L v((n,L))}\sum_{\ell | L^{\infty}} \ell^{-1} \Delta_{k,R}(m \ell^2,n).
\end{equation}

The identity~\eqref{ILS} is not exactly sufficient for our purpose when we shall average over families $\Sprim$ with $N$ square-free in \S\ref{sec:sfree}. Indeed the condition $(n,N^2)\mid N$ is too restrictive unless $N$ is prime (\S\ref{sec:pf}). Therefore, we shall use the following variant in~\S\ref{sec:sfree}. As will be clear from the proof, this variant is already present in the work of Iwaniec, Luo and Sarnak as it is a particular case of~\cite{ILS00}*{Eq.~(2.51)}.
\begin{lemma}\label{l:ILS-modified}
 Suppose $N$ is square-free and $m,n$ are positive integers with $(m,N)=1$. Then
\begin{equation*} 
	\Delta_{k,N}^*(m,n)=
	\sum_{\substack{LR=N}} 
\frac{\mu(L)}{L \; v( (n,L))}
\sum_{\ell\mid L^\infty}\ell^{-1}
\sum_{\ell_1^2\mid (n,\ell_1L)}
\mu(\ell_1)\ell_1
\Delta_{k,R}\left(
m \ell^2,\frac{n}{\ell_1^2}
\right).
\end{equation*}
\end{lemma}
\begin{remark}\label{ILSRemarks}

\begin{enumerate}
  \item Under the assumption $(n,N^2)\mid N$, only the term $\ell_1=1$ contributes. We recover the formula~\eqref{ILS}.  
  \item The level $R$ is coprime with $\ell$ and $\ell_1$. This is of great importance in subsequent estimates when applying the Petersson trace formula which yields Kloosterman sums $S(m \ell^2,\frac{n}{\ell_1^2};c)$ with $c\equiv 0(R)$. The quality of obtained bounds diminishes as $(R,\ell)$ or $(R,\ell_1)$ increases. 
  \item For those $L>1$, one may truncate the $\ell$-sum to $\ell \leqslant L^A$ for any $A>0$ up to an error of size $O_\varepsilon \left((nm)^{\varepsilon}L^{-A}\right)$. \label{elltruncation}
\end{enumerate}
\end{remark}

\begin{proof} 
We begin with the formula in~\cite{ILS00}*{Eq.~(2.51)}
\begin{equation}
 \Delta_{k,N}(m,n)= \frac{12}{(k-1) N}
\sum_{LR=N}
\sum_{f\in \mathcal{S}^*_k(R)}
A_f(m,L)A_f(n,L)\frac{Z_N(1,f)}{Z(1,f)} \label{ILSstart}
\end{equation}
valid for $N$ square-free with $(n,m,N)=1$ where 
\begin{equation}
  Z(1,f)^{-1}=(4\pi)^{1-k} \Gamma(k) \frac{v(N)\phi(R)}{12 R} \langle f,f\rangle^{-1}_N \label{Z1f}
\end{equation}
as in \cite{ILS00}*{Lemma~2.5}, 
\begin{equation}
  Z_N(1,f):=\sum_{\ell|N^{\infty}} \lambda_f(\ell^2) \ell^{-s} \label{ZN1f}
\end{equation}
as in \cite{ILS00}*{Eq.~2.49} and
\begin{equation}
A_f(n,L):=\frac{1}{v( (n,L))}\sum_{\ell_1^2\mid (n,\ell_1L)}^{}
\mu(\ell_1)\ell_1\lambda_f\left( \frac{n}{\ell_1^2} \right) \label{AfnL}
\end{equation}
as in \cite{ILS00}*{Eq.~2.47}. 

 We observe by \eqref{Z1f} and \eqref{vN} that $ \langle f,f\rangle_N=v(L) \langle f,f\rangle_R$. Through our convention $\omega_f=\dfrac{(4\pi)^{k-1}}{\Gamma(k-1)}\langle f,f\rangle_R$ when $f \in \mathcal{S}_k^*(R)$, one can write \eqref{ILSstart} as
$$
 \Delta_{k,N}(m,n)= 
\sum_{LR=N} \frac{\phi(R) v(N)}{v(L) R N}
\sum_{f\in \mathcal{S}^*_k(R)}
\omega_f^{-1} A_f(m,L)A_f(n,L)Z_N(1,f)
$$
which is the same as
$$
 \Delta_{k,N}(m,n)= 
\sum_{LR=N} \frac{1}{L} \prod_{p|R} \left(1-\frac{1}{p^2}\right) 
\sum_{f\in \mathcal{S}^*_k(R)}
\omega_f^{-1} A_f(m,L)A_f(n,L)Z_N(1,f).
$$
Now since $(m,N)=1$, one may use \eqref{AfnL},  \eqref{ZN1f} and \eqref{SpecialHecke} to write this as
$$
 \Delta_{k,N}(m,n)= 
\sum_{LR=N} \frac{1}{L\;v((n,L))} \sum_{\ell | L^{\infty}} \frac{1}{\ell}\sum_{\ell_1^2\mid (n,\ell_1L)}^{}
\mu(\ell_1)\ell_1
\sum_{f\in \mathcal{S}^*_k(R)}
\omega_f^{-1} \lambda_f(m \ell^2) \lambda_f\left(\frac{n}{\ell_1^2}\right)
$$
and the Lemma follows by an application of M\"{o}bius inversion.

\end{proof}

\subsection{A bound on smooth numbers}
In the proof for square-free $N$ in \S\ref{sec:sfree} we shall need the following.
\begin{lemma}\label{l:smooth}
  Fix a positive real $A>0$. For any square-free integer $L\ge 1$, the number of integers $\ell\le L^A$ with $\ell \mid L^\infty$ is $ L^{o(1)}$.
\end{lemma}
\begin{proof}
  This is an elementary adaptation of the Rankin method. 
More precise estimates and asymptotics are discussed in~\cite{book:MV-mult}*{Chap.~7}.
\footnote{In the worst case situation where $L$ is the product of all primes $p\le y$ this a bound for $\Psi(x,y)$ with $y\asymp \log x$ and $x=L^A$ which may be shown to be at most $\exp(c\frac{\log x}{ \log \log x})=L^{o(1)}$.}  
  Let $\sigma>0$; the number of integers we are estimating is
\begin{equation*}
  \begin{aligned}  
	\#\{\ell\le L^A,\ \ell\mid L^\infty\}	\le&
	\sum_{\ell \mid L^\infty} \left(\frac{L^A}{\ell}\right)^{\sigma}
	\le L^{A\sigma} \prod_{\ell \mid L} \left( 1-\frac{1}{\ell^{\sigma}} \right)^{-1}\\
	\ll & L^{A\sigma} (1-2^{-\sigma})^{-\omega(L)} \ll_{\epsilon,\sigma} L^{A\sigma+\epsilon}.
  \end{aligned}
\end{equation*}
Choosing $\sigma>0$ arbitrarily small concludes the claim.
\end{proof}

\subsection{Averaging over Hecke eigenvalues with additive twists}
The following is established in~\cite{KMV02}*{Appendix~A}.
\begin{lemma}[Vorono\"{i} summation]
  \label{lem:voronoi}
  (i) Let $(a,c)=1$ and let $h$ be a smooth function, compactly supported in $(0,\infty)$. Let $g\in S_\kappa^*(M,\chi)$ be a holomorphic newform of level $M$. Set $M_2:=M/(M,c)$. Then there exists a complex number $\eta$ of modulus $1$ (depending on $a,c$ and $g$) and a newform $g^\ast \in S_\kappa^*(M,\chi^*)$ of the same level $M$ and the same archimedean parameter $\kappa$ such that
$$
\sum_n \lambda_g(n) e\bigg(n\frac{a}{c}\bigg)h(n) =  \frac{2 \pi\eta}{c\sqrt{M_2}}\sum_n \lambda_{g^{\ast}}(n) e\bigg(-n\frac{\overline{aM_2}}{c}\bigg)\int_0^{\infty} h(\xi) J_{\kappa-1}\bigg(\frac{4\pi\sqrt{n\xi}}{c\sqrt{M_2}}\bigg) d\xi
$$
where $\overline{x}$ denotes the multiplicative inverse of $x$ modulo $c$.

(ii) Decomposing $\chi=\chi_1 \chi_2$ into characters of modulus $M_1$ and $M_2$ respectively, where $M=M_1M_2$, we have that $\chi^*=\chi_1 \overline{\chi_2}$.

(iii) Let $\eta_g(M_2)$ be the pseudo-eigenvalue of $g$ for the $M_2$ Atkin--Lehner operator. Then
\begin{equation} 
  \eta = i^{\kappa} \overline{\chi_1(a)}
  \chi_2(-c) \eta_g(M_2).
\end{equation}

(iv) The Hecke eigenvalues of $g^*$ are given by
\begin{equation} 
  \lambda_{g^*}(n)=
  \begin{cases}
	\overline{\chi_2(n)} \lambda_g(n), & \text{if $(n,M_2)=1$,}\\
	\chi_1(n) \overline{\lambda_g(n)}, & \text{if $n\mid M_2^{\infty}$.}
  \end{cases}
\end{equation}
\end{lemma}

\section{Proof of subconvexity when $N$ is prime}\label{sec:pf} 
Let $1<N<M$ with $N$ prime and $M$ square-free such that $(N,M)=1$. Let $f\in \Sprim$ and let $g\in \SpM$ with $g$ self-dual. As discussed in \S\ref{sec:intro:positivity}, this implies that $\chi$ is quadratic and that the root number is $\epsilon(f \times g)= \pm 1$.   
In the case that $\epsilon(f \times g)= 1$, the approximate functional equation argument reduces our $L$-function to the analysis of the sum:
$$
L(\tdemi,f\times g)=
2\sum_n
\frac{\lambda_f(n)\lambda_g(n)}{\sqrt{n}}
V\Bigl(
\frac{n}{NM}
\Bigr) 
$$
where $V$ satisfies the properties of \S \ref{sec:pre:app}.  Recall further that such an object is known to be non-negative when $f \times g$ is self-dual symplectic and therefore we assume in the end that $g$ is dihedral in order to establish subconvexity as a corollary to our first moment bound.  To set up for our application of~\eqref{ILS}, we trivially write the above as
\begin{equation} \label{firstprime}
2\sum_{(n,N^2)|N}
\frac{\lambda_f(n)\lambda_g(n)}{\sqrt{n}}
V\Bigl(
\frac{n}{NM}
\Bigr) +O\left(\frac{\sqrt{M}}{N}\right)
\end{equation}
for any $f\in \Sprim$.  Therefore, a first moment average over central $L$-values reduces to the study of
$$
S=\sum_{f\in \Sprim} \omega_f^{-1} \sum_{(n,N^2)|N} \frac{\lambda_f(n)\lambda_g(n)}{\sqrt{n}}
V\Bigl(
\frac{n}{NM}
\Bigr)
$$
up to an error of size $\ll \sqrt{M} N^{-1}$.
\subsection{From newforms to full bases of Hecke eigenforms}
We start by changing the order of summation in $S$ above
$$
S=\sum_{(n,N^2)|N} \frac{\lambda_g(n)}{\sqrt{n}}
V\Bigl(
\frac{n}{NM}
\Bigr)\Delta^\ast_{k,N}(1,n).
$$
Using the fact that $N$ is prime, we convert the sum over newforms by~\eqref{ILS} to sums over Hecke eigenbases for $\mathcal{S}_k(N)$ and $\mathcal{S}_k(1)$  
$$
 \sum_{(n,N^2)|N} \frac{\lambda_g(n)}{\sqrt{n}}
V\Bigl(
\frac{n}{NM}
\Bigr)\left\{\Delta_{k,N}(1,n)-\frac{1}{Nv((n,N))}\sum_{j\geqslant 0} N^{-j} \Delta_{k,1}(N^{2j},n)\right\}.
$$
By trivial estimates, one sees that the terms
$$
 \sum_{n\equiv 0(N)} \frac{\lambda_g(n)}{\sqrt{n}}
V\Bigl(
\frac{n}{NM}
\Bigr)\frac{1}{Nv((n,N))}\sum_{j\geqslant 0} N^{-j} \Delta_{k,1}(N^{2j},n)
$$
and
$$
 \sum_{(n,N)=1} \frac{\lambda_g(n)}{\sqrt{n}}
V\Bigl(
\frac{n}{NM}
\Bigr)\frac{1}{Nv((n,N))}\sum_{j\geqslant 1} N^{-j} \Delta_{k,1}(N^{2j},n)
$$
are both of size $\ll \sqrt{M} N^{-1}$.  Furthermore, the sums
$$
 \sum_{n\equiv 0(N)} \frac{\lambda_g(n)}{\sqrt{n}}
V\Bigl(
\frac{n}{NM}
\Bigr)\frac{1}{N}\Delta_{k,1}(1,n)
$$
and
$$
 \sum_{n\equiv 0(N^2)} \frac{\lambda_g(n)}{\sqrt{n}}
V\Bigl(
\frac{n}{NM}
\Bigr)\Delta_{k,N}(1,n)
$$
satisfy the same bound and may be added back to $S$. Therefore, we see that
\begin{equation}
S=\sum_{n} \frac{\lambda_g(n)}{\sqrt{n}}
V\Bigl(
\frac{n}{NM}
\Bigr)\left\{\Delta_{k,N}(1,n)-\frac{1}{N}\Delta_{k,1}(1,n)\right\} + O\left(\frac{\sqrt{M}}{N} (NM)^{\varepsilon}\right).\label{newminusold}
\end{equation}
One can simply view the above as rewriting the spectral average over newforms of level $N$ as a sum over all forms of level $N$ minus the contribution of the old forms. 
\subsection{Old form contribution}\label{sec:pf:old}
Since the weights of our forms are fixed, we can think of $\Delta_{k,1}(1,n)$ in \eqref{newminusold} as a fixed form of full level. One has 
$$
\frac{1}{N}\sum_{n} \frac{\lambda_g(n)}{\sqrt{n}}
V\Bigl(
\frac{n}{NM}
\Bigr)\Delta_{k,1}(1,n) = \frac{1}{N} \sum_{h\in \mathcal{B}_k(1)} \omega_h^{-1}\sum_{n} \frac{\lambda_g(n)\psi_h(n)}{\sqrt{n}}
V\Bigl(
\frac{n}{NM}
\Bigr)
$$
so that we must analyze
$$
\frac{1}{2\pi i}
 \int_{(2)}
\left(NM\right)^{s} \frac{\Gamma(s+\frac{1+\abs{k-\kappa}}{2})\Gamma(s+\frac{k+\kappa-1}{2})}{\Gamma(\frac{1+\abs{k-\kappa}}{2})\Gamma(\frac{k+\kappa-1}{2})}L(s+\tfrac{1}{2}, g \times h) L^{(N)}(1+2s,\chi) G(s)
ds
$$
where $G(s)$ satisfies the properties of \S \ref{sec:pre:app}.  Shifting the contour to the left, one picks up the contribution from the pole at $s=0$ and then applies the convexity bound for the Rankin-Selberg $L$-function $L(\tfrac12,g\times h)$ to obtain
$$
\frac{1}{N} \sum_{n} \frac{\lambda_g(n)}{\sqrt{n}}
V\Bigl(
\frac{n}{NM}
\Bigr)\Delta_{k,1}(1,n) \ll_{\varepsilon} \frac{M^{1/2+\varepsilon}}{N}.
$$

Therefore, the contribution from the old forms is absorbed into the error term and we have that
\begin{equation}
S=   \sum_{n} \frac{\lambda_g(n)}{\sqrt{n}}
V\Bigl(
\frac{n}{NM}
\Bigr)\Delta_{k,N}(1,n)+O_{\varepsilon}\left(\frac{\sqrt{M}}{N}(NM)^{\varepsilon}\right). \label{reduction}
\end{equation}

\begin{remark}
We have arrived at the following.
For $N$ prime we have
\begin{equation}
\sum_{n} \frac{\lambda_g(n)}{\sqrt{n}}
V\Bigl(
\frac{n}{NM}
\Bigr)\left\{\Delta^\ast_{k,N}(1,n) -\Delta_{k,N}(1,n)\right\}\ll_{\varepsilon}\frac{\sqrt{M}}{N}(NM)^{\varepsilon}\label{averagedifference}.
\end{equation}
\end{remark}
\subsection{The average over $\mathcal{S}_k(N)$}
By equation \eqref{reduction}, we are left with treating the full first moment average
$$
\sum_{n} \frac{\lambda_g(n)}{\sqrt{n}}
V\Bigl(
\frac{n}{NM}
\Bigr)\Delta_{k,N}(1,n)
$$
as we did in the outline \S\ref{s:sketch} with a few additional details.

Petersson's trace formula produces a diagonal term and off-diagonal terms of the form
$$
\sum_{n}\frac{\lambda_g(n)}{\sqrt{n}}V\Bigl(
\frac{n}{NM}
\Bigr)\left\{\delta(n,1)+2\pi i^{-k}\sum_{\substack{c>0\\c\equiv 0(N)}}\frac{1}{c}S(n,1;c) J_{k-1}\left(\frac{4\pi\sqrt{n}}{c}\right)\right\}
$$
where $\delta(1,1)=1$, $\delta(n,1)=0$ otherwise.   Therefore, the diagonal term contribution to our full first moment average is simply $V(1/NM)$.

We now turn our attention to the off-diagonal sums. 
We start by taking a smooth partition of unity for $V$ and consider sums over dyadic segments, of different sizes $N \leqslant Z\leqslant (NM)^{1+\varepsilon}$ say, controlled by some positive, smooth function $h$ supported in $[1/2,5/2]$. The standard Bessel function bounds in \S\ref{sec:pre:bessel}, along with an application of Weil's bound for Kloosterman sums, allows us to truncate our $c$-sum to one of length $c\leqslant Z^A$ up to an error term of size $Z^{-B}N^{-1}$ for some positive $A$ and $B$ (note that $k\ge 2$). For the remaining sum over $c$, we change the order of summation and break apart the $c$-sum relative to $(c,M)$ in order  to have 
\begin{equation}
\sum_{M_2|M}\sum_{\substack{N\leqslant c\leqslant Z^A \\c\equiv 0(N)\\ (c,M)=M/M_2}}\frac{1}{c}\sum_{n}\frac{\lambda_g(n)}{\sqrt{n}}S(n,1;c) J_{k-1}\left(\frac{4\pi\sqrt{n}}{c}\right)h\left(\frac{n}{Z}\right). \label{endoftruncations}
\end{equation}
For each fixed $c$ in the outer sums, an application of Lemma~\ref{lem:VoronoiApp} to the $n$-sum with $D=1$ and $q=c$ bounds the above by
\begin{eqnarray*}
& &\sum_{M_2|M}\sum_{\substack{N\leqslant c\leqslant Z^A \\c\equiv 0(N)\\ (c,M)=M/M_2}}\frac{1}{c} \left(\sqrt{\frac{Z}{M_2}}+ c\bigg(1+\frac{\sqrt{Z}}{c}\bigg)\sqrt{\frac{M_2}{Z}}\right) \frac{(\sqrt{Z}/c)^2}{(1+\sqrt{Z}/c)^3} (MZ)^{\varepsilon}\\
& \ll & \sum_{M_2|M}\left\{\sum_{\substack{N\leqslant c\leqslant \sqrt{Z}\\c\equiv 0(NM/M_2)}}\frac{1}{\sqrt{M_2}}\left(1+\frac{M_2}{\sqrt{Z}}\right)+\sum_{\substack{\sqrt{Z} < c\leqslant Z^A \\ c\equiv 0(NM/M_2)}}\frac{\sqrt{Z}}{c^2}\left(\frac{Z}{c\sqrt{M_2}}+\sqrt{M_2}\right)\right\}(MZ)^{\varepsilon}\\
& \ll & \left(\frac{\sqrt{Z}}{N\sqrt{M}}+\frac{\sqrt{M}}{N}\right)(MZ)^{\varepsilon} \ll \frac{\sqrt{M}}{N} (NM)^{\varepsilon}.
\end{eqnarray*}

Therefore, one establishes that 
 $$
\sum_{n} \frac{\lambda_g(n)}{\sqrt{n}}
V\Bigl(
\frac{n}{NM}
\Bigr)\Delta_{k,N}(1,n)= V\Bigl(\frac{1}{NM}\Bigr)+O_{\varepsilon}\left(\frac{\sqrt{M}}{N}(NM)^{\varepsilon}\right)
$$
after summing over all dyadic segments $Z$ in $n$. By \eqref{reduction} and \eqref{firstprime} we get that

\begin{equation}
  \label{pf-summary} 
\sum_{f\in \Sprim} \omega_f^{-1} \sum_{n}\frac{\lambda_f(n)\lambda_g(n)}{\sqrt{n}} V\left(\frac{n}{NM}\right)= V\Bigl(\frac{1}{NM}\Bigr)
+O_{\varepsilon}\left(\frac{\sqrt{M}}{N}(NM)^{\varepsilon}\right).
\end{equation}
Therefore, assuming that $g$ is also dihedral, such that positivity of all central $L$-values is known in the average over $f\in \Sprim$, one has by Lemma \ref{lem:V} the subconvexity bound
\begin{equation}
L( \tfrac{1}{2}, f\times g) \ll_\varepsilon \sqrt{NM}\left(\frac{\sqrt{N}}{\sqrt{M}}+\frac{1}{\sqrt{N}}\right)(NM)^{\varepsilon} \label{FinalSubconvex}
\end{equation}
for any individual $L$-function in our family.

\section{Proof of Theorem~\ref{th:intro}}\label{sec:sfree}
Let $1\leqslant N<M$ with $N$ and $M$ square-free such that $(N,M)=1$.  For $g\in \SpM$ we consider the first moment
$$
\sum_{f\in \Sprim} \omega_f^{-1} L^{(j)}(\tfrac{1}{2}+it,f\times g).
$$
Using the approximate functional equation in~\S\ref{sec:pre:app} we reduce the analysis of the average of central $L$-values to the analysis of 
\begin{equation} \label{S}
S=\sum_{f\in \Sprim} \omega_f^{-1} \sum_{n} \frac{\lambda_f(n)\lambda_g(n)}{\sqrt{n}}
V\Bigl(
\frac{n}{NM}
\Bigr) = \sum_{n} \frac{\lambda_g(n)}{\sqrt{n}}
V\Bigl(
\frac{n}{NM}
\Bigr)\Delta^\ast_{k,N}(1,n)
\end{equation}
where $V$ satisfies the properties of \S \ref{sec:pre:app}. In this reduction we make crucial use of the fact that the root number $\epsilon(f\times g)$ is independent of $f\in \Sprim$ (see~\S\ref{sec:pre:RS}).  


\subsection{Averaging over newforms}
The average $S$ in \eqref{S} over newforms of level $N$ can be written as follows using Lemma~\ref{l:ILS-modified}
\begin{equation*}
  \sum_{n}
\frac{\lambda_g(n)}{\sqrt{n}}
V\Bigl(
\frac{n}{NM}
\Bigr)
\sum_{LR=N}
\frac{\mu(L)}{L v( (n,L))}
\sum_{\ell \mid L^\infty}
\ell^{-1}
\sum_{\ell_1^2\mid (n,\ell_1L)}
\mu(\ell_1)\ell_1
\Delta_{k,R}(\ell^2,\frac{n}{\ell_1^2}).
\end{equation*}
Furthermore, by Remark \ref{ILSRemarks} (\ref{elltruncation}), one may restrict to considering only those $\ell |L^{\infty}$ with $\ell \leqslant L^A$ for some large $A>0$. We change the order of summation to obtain
\begin{equation*}
\sum_{LR=N}
\frac{\mu(L)}{L}
\sum_{\substack{\ell\mid L^\infty \\ \ell \leqslant L^A}}
\ell^{-1}
\sum_{\ell_1\mid L}
\mu(\ell_1)
\sum_{n}
\frac{\lambda_g(n\ell_1^2)}{v( (n\ell_1^2,L))\sqrt{n}}
V\Bigl(
\frac{n\ell_1^2}{NM}
\Bigr)
\Delta_{k,R}(\ell^2,n).  
\end{equation*}
Letting $L_1=\frac{L}{\ell_1}$, we note that $v( (n\ell_1^2,L))=v(\ell_1)v( (n,L_1))$. Thus it remains to focus on the inner sum
\begin{equation*}
S_{\textnormal{in}}=\sum_{n}
\frac{\lambda_g(n\ell_1^2)}{v( (n,L_1))\sqrt{n}}
V\Bigl(
\frac{n\ell_1^2}{NM}
\Bigr)
\Delta_{k,R}(\ell^2,n).    
\end{equation*}

\subsection{Averaging over all forms of level $R$}

We apply Petersson's trace formula. The diagonal contribution is then given by
$$
\frac{\lambda_g(\ell^2\ell_1^2)}{v( (\ell^2,L_1))\ell}
V\Bigl(
\frac{\ell^2 \ell_1^2}{NM}
\Bigr)  
$$
so that the total diagonal contribution to $S$ is
\begin{equation}
\sum_{LR=N}
\frac{\mu(L)}{L}
\sum_{\substack{\ell\mid L^\infty \\ \ell \leqslant L^A}}
\ell^{-2}
\sum_{\ell_1\mid L}
\mu(\ell_1)
\frac{\lambda_g(\ell^2\ell_1^2)}{v( (\ell^2\ell_1^2,L))}
V\Bigl(
\frac{\ell^2\ell_1^2}{NM}
\Bigr) \ll (NM)^\varepsilon. \label{diagonalboundsquarefree}
\end{equation}
Ignoring the $2\pi i^{-k}$ factor, the off-diagonal terms in $S_{\textnormal{in}}$ are
\begin{equation*}
\sum_{n}
\frac{\lambda_g(n\ell_1^2)}{v( (n,L_1))\sqrt{n}}
V\Bigl(
\frac{n\ell_1^2}{NM}
\Bigr)
\sum_{c\equiv 0(R)}
\frac{S(\ell^2,n;c)}{c}
J_{k-1}\left( \frac{4\pi \ell\sqrt{n}}{c} \right).
\end{equation*}

It is convenient to apply Selberg's identity to the Kloosterman sums so that
\begin{equation*}
S(\ell^2,n;c) = \sum_{\ell_2|(\ell^2,n,c)} \ell_2 S\Big(\frac{n\ell^2}{\ell_2^2},1;\frac{c}{\ell_2}\Big).
\end{equation*}
We then let $q=\frac{c}{\ell_2}$ and pull out the new $\ell_2$ factors from $n$ and $\ell^2$. To take care of the term $v( (n\ell_2,L_1))=v( (\ell_2,L_1))v( (n,L_2))$ where $L_2:=\frac{L_1}{(\ell_2,L_1)}$, we set $\ell_3=(n,L_2)$ and use M\"obius inversion:
\begin{equation*}
  \sum_{\ell_2\mid \ell^2} \frac{1}{v( (\ell_2,L_1))\sqrt{\ell_2}}
  \sum_{\ell_3\mid L_2} \frac{1}{v(\ell_3)\sqrt{\ell_3}}
  \sum_{\ell_4\mid L_2}\frac{\mu(\ell_4)}{\sqrt{\ell_4}}
\sum_{n}
\frac{\lambda_g(n\ell_1^2\ell_2\ell_3\ell_4)}{\sqrt{n}}
V\Bigl(
\frac{n\ell_1^2\ell_2\ell_3\ell_4}{NM}
\Bigr)
\Delta^{\text{off}}_{k,R}(\frac{n\ell^2\ell_3\ell_4}{\ell_2},1),    
\end{equation*}
where we have, using the fact that $(\ell_2,R)=1$,
\begin{equation*}
\Delta^{\text{off}}_{k,R}(n,1)=
\sum_{q\equiv 0(R)}
\frac{1}{q} S\Big(n,1;q\Big)
J_{k-1}\left( \frac{4\pi \sqrt{n}}{q} \right).
\end{equation*}

We apply the Hecke relation
\begin{equation*}
  \lambda_g(n\ell_1^2\ell_2\ell_3\ell_4)
  =\sum_{\ell_5\mid (n,\ell_1^2\ell_2\ell_3\ell_4)}
  \mu(\ell_5)\chi(\ell_5)
  \lambda_g(\frac{n}{\ell_5})
  \lambda_g(\frac{\ell_1^2\ell_2\ell_3\ell_4}{\ell_5}).
\end{equation*}
The inner $n$-sum may therefore be rewritten as 
\begin{equation*}
  \sum_{\ell_5\mid \ell_1^2\ell_2\ell_3\ell_4}
 \mu(\ell_5)\chi(\ell_5)
 \frac{\lambda_g(\frac{\ell_1^2\ell_2\ell_3\ell_4}{\ell_5})}{\sqrt{\ell_5}}
 \sum_{n}
  \frac{\lambda_g(n)}{\sqrt{n}}
V\Bigl(
\frac{n\ell_1^2\ell_2\ell_3\ell_4\ell_5}{NM}
\Bigr)
\Delta^{\text{off}}_{k,R}(\frac{n\ell^2\ell_3\ell_4\ell_5}{\ell_2},1)
.
\end{equation*}

We let $D=\frac{\ell^2\ell_3\ell_4\ell_5}{\ell_2}$. Note that $(D,R)=1$. Breaking the $n$-sum into dyadic segments of length $Z\leqslant (\frac{NM}{\ell_1^2\ell_2\ell_3\ell_4\ell_5})^{1+\varepsilon}$ through a smooth partition of unity with $h$ a smooth function compactly supported on $[1/2,5/2]$, it remains to estimate the inner sums over $q$ and $n$ which may be written as
\begin{equation*}
  \sigma_Z:=
  \sum_{q\equiv 0(R)} \frac{1}{q}
 \sum_{n}
  \frac{\lambda_g(n)}{\sqrt{n}}
S\Big(n D,1;q\Big)
  J_{k-1}\left( \frac{4\pi \sqrt{nD}}{q} \right)
h\left(\frac{n}{Z}\right).
\end{equation*}
As in the prime level case of \S\ref{sec:pf}, the Bessel function bounds in \S\ref{sec:pre:bessel} along with an application of Weil's bound for Kloosterman sums, allows us to truncate our $q$-sum to one of length $q\leqslant (DZ)^B$ up to an error term of size $(DZ)^{-C}R^{-1}$ for some positive $B$ and $C$.  An application of Lemma~\ref{lem:VoronoiApp} gives that 
\begin{equation*}
  \sigma_Z  \ll \sum_{M_2|M} \sum\limits_{\substack{R \leqslant q \leqslant (DZ)^B \\ q\equiv 0 (R) \\ (q,M)=M/M_2}} 
  \frac{1}{q} \frac{P^2}{(1+P)^{3}} \left(\sqrt{\dfrac{Z}{M_2}}+q_2(1+P) \sqrt{\dfrac{M_2}{Z}}
  \right)(DMZ)^{\varepsilon}
\end{equation*}
with $P=\frac{\sqrt{DZ}}{q}$ and $q_2=\frac{q}{(D,q)}$. Note, we have used that $(M,D)=1$ which implies $(M,q_2)=(M,q)$.

 As in the prime level case, we now split the above inner $q$-sum into two parts based on the size of $P$ relative to $1$. Thus, the sum $\sigma_Z$ is bounded by
\begin{equation*}	
\begin{aligned}
&\sum_{M_2|M} \left\{\sum_{\substack{q \leqslant \sqrt{DZ} \\ q\equiv 0 (R M/M_2) }} \frac{1}{qP} 
	\left(\sqrt{\dfrac{Z}{M_2}}+q_2 P \sqrt{\dfrac{M_2}{Z}}
  \right)
  +\sum_{\substack{\sqrt{DZ} \leqslant q \leqslant (DZ)^B\\ q\equiv 0 (RM/M_2)}}\frac{P^2}{q}
  \left(\sqrt{\dfrac{Z}{M_2}}+q_2 \sqrt{\dfrac{M_2}{Z}}
  \right)\right\}(DMZ)^{\varepsilon}
  \\
	\ll & 
	  R^{-1} \left(  \frac{\sqrt{Z}}{\sqrt{M}} + \sqrt{DM}  \right) (DMZ)^{\varepsilon} \ll  R^{-1} \left(  \frac{\sqrt{N}}{\ell_1\sqrt{\ell_2 \ell_3 \ell_4 \ell_5}} + \sqrt{DM}  \right)(DNM)^{\varepsilon}.
\end{aligned}
\end{equation*}
 
Furthermore, we recall that $D=\frac{\ell^2\ell_3\ell_4\ell_5}{\ell_2}$. Therefore, combining the above, it remains to treat the sum over the $\ell_\bullet$, that is: 
\begin{equation*}
  \sum_{\substack{\ell\mid L^\infty \\ \ell \leqslant L^A}} \ell^{-1}
  \sum_{\ell_1\mid L}
  v(\ell_1)^{-1}
  \sum_{\ell_2 \mid \ell^2} 
  \frac{1}{v( (\ell_2,L_1))\sqrt{\ell_2}}  
    \sum_{\ell_3\mid L_2}
  \frac{1}{v(\ell_3)\sqrt{\ell_3}}
  \sum_{\ell_4\mid L_2}
  \frac{1}{\sqrt{\ell_4}}
  \sum_{\ell_5\mid \ell_1^2\ell_3\ell_4\ell_2} \frac{\sqrt{D}}{\sqrt{\ell_5}}.
\end{equation*}
By trivial estimates and an application of Lemma \ref{l:smooth} one bounds the remaining terms above by $N^{\varepsilon}$ for each $L|N$.  Therefore, we see that our first moment average satisfie
\begin{equation*}
S \ll \left(1+ \sum_{LR=N} \frac{1}{LR} \left(\sqrt{N}+\sqrt{M}\right)\right)(NM)^{\varepsilon} \ll \left(1 + \frac{\sqrt{M}}{N}\right) (NM)^{\varepsilon}.
\end{equation*}

\section{Proof of Lemma~\ref{lem:VoronoiApp}}
\label{sec:pflem}
Let $Z\ge 1$ and let $k$ and $\kappa$ be fixed positive integers.  Let $q, D, M$ be positive integers with $M$ square-free.  Let $g\in \SpM$ be a newform of weight $\kappa$, level $M$ and nebentypus~$\chi$.  In this section, we consider sums of the form
\begin{equation}\label{SZD-precise}
 S_Z(D; q):=\sum_{n}
  \frac{\lambda_g(n)}{\sqrt{n}}
S\Big(n D,1;q\Big)
  J_{k-1}\left( \frac{4\pi \sqrt{nD}}{q} \right)
h\left(\frac{n}{Z}\right)
\end{equation}
where $h$ is a smooth function, compactly supported in $[1/2,5/2]$ with bounded derivatives. The aim is to establish Lemma~\ref{lem:VoronoiApp}.

Setting $D_1 := (D,q)$ and $D_2:=D/D_1$, we note that $S(nD,1;q)=S(nD_2D_1,1;q)$ so that $S_Z(D;q)$ equals
$$
 \sideset{}{^\ast}\sum_{\alpha(q)} e\left(\frac{\alpha}{q}\right)\sum_{n}
  \frac{\lambda_g(n)}{\sqrt{n}}e\left(n\frac{\bar{\alpha}D_2}{q/D_1}\right)J_{k-1}\left( \frac{4\pi \sqrt{nD}}{q} \right)
h\left(\frac{n}{Z}\right).
$$
We are now in position to apply the Vorono\"i formula (Lemma~\ref{lem:voronoi}) to the inner sum over $n$.  Note that $\bar{\alpha}$ was first chosen such that $\alpha \bar{\alpha} \equiv 1(q)$ so that we also have $\alpha \bar{\alpha}  \equiv 1(q/D_1)$.
Set $M_2:=\frac{M}{(q/D_1,M)}$ (we don't assume $(D,M)=1$ so it could be that $(M_2,q)>1$ but this won't affect the argument). The inner $n$-sum becomes, up to some bounded multiplicative factors
\begin{equation} 
  \frac{D_1}{q \sqrt{M_2}}
  \sideset{}{^\ast}  \sum_{\alpha (q)}
  e\left(\frac{\alpha}{q}\right)
  \sum_{n}
  \lambda_{g^*}(n)
  e\left( -n\frac{\alpha \overline{D_2 M_2}}{q/D_1} \right) I_q(n) \label{afterVoronoi}
\end{equation}
where $M_2 \overline{M_2} \equiv 1 (q/D_1)$, $D_2 \overline{D_2} \equiv 1 (q/D_1)$ and
\begin{equation*} 
  I_q(n)= 
  \int_{0}^{\infty}
  h\left(\frac{x}{Z}\right)x^{-\Mdemi}
  J_{k-1}
  \left( \frac{4\pi \sqrt{xD}}{q} \right)
  J_{\kappa -1}\left( \frac{4\pi D_1 \sqrt{nx}}{q \sqrt{M_2}} \right)dx.
\end{equation*}
The sum over $\alpha$ may be written as
$$
 \sideset{}{^\ast}   \sum_{\alpha (q)}
 e\left( \frac{\alpha }{q}- \frac{n\alpha \overline{D_2M_2}}{q/D_1} \right)=S(0,1-nD_1\overline{D_2M_2};q),
$$
a Ramanujan sum. Therefore, \eqref{afterVoronoi} is reduced to
\begin{equation}
  \label{arithprog}
  \frac{D_1}{\sqrt{M_2}}\sum_{q_1|q}^{}
  \frac{\mu(q_1)}{q_1} \sum_{\substack{n\ge 1\\ D_1\overline{D_2M_2} n \equiv  \;(q/q_1)}}
  \lambda_{g^*}(n) I_q(n).
\end{equation}
We see that necessarily $D_1$ is coprime with $q/q_1$ otherwise the congruence cannot be satisfied. This implies that $D_1\mid q_1$ and $D_1|| q$. Then $D_2$ is coprime with $q/q_1$ and also because of the M\"obius function, $D_1$ is square-free. These conditions may also be seen by inspecting the Kloosterman sum $S(nD,1;q)$ we have at the beginning in~\eqref{SZD-precise}.

We are left with bounding a weighted sum of Hecke eigenvalues over an arithmetic progression. A change of variables in the integral in \eqref{arithprog} shows that
$$
I_q(n)=\sqrt{Z} \; I\left(\frac{\sqrt{DZ}}{q},\frac{D_1\sqrt{nZ}}{q\sqrt{M_2}}\right)$$
with $I(a,b)$ as in the statement of Lemma~\ref{lem:Icn}.
Recall that $P:=\sqrt{DZ}/q$.  The inner sum over $n$ may therefore be restricted, up to a negligible error term, by
\begin{equation} \label{nset}
  D_1n\equiv D_2 M_2 \pmod{q/q_1}, \quad
  \abs{1-\sqrt{\frac{D_1n}{D_2M_2}}}\ll P^{-1} (qDMZ)^{2\varepsilon}
\end{equation}
with $I_q(n)$ bounded by $\sqrt{Z} \dfrac{P^2}{(1+P)^3}$ in that range.  The number of $n$ satisfying \eqref{nset} is 
$$
\ll_\varepsilon \left(1+\dfrac{D_2 M_2}{D_1}\dfrac{(1+P)}{P^2}\dfrac{q_1}{q}\right)(qDMZ)^{\varepsilon}.
$$ 
Thus, one establishes the final bound (recall that $D_1|q_1$)
$$
S_Z(D;q) \ll_\varepsilon \frac{\sqrt{Z}}{\sqrt{M_2}} \dfrac{P^2}{(1+P)^3}\left(1+\frac{D M_2}{D_1 q}\frac{(1+P)}{P^2}\right)(qDMZ)^{\varepsilon}.
$$
which may also be written as in the statement of Lemma~\ref{lem:VoronoiApp}.


\subsection*{Acknowledgments}  A first version of this manuscript has been worked out at the Institute for Advanced Studies during the special year 2009/2010 on Analytic Number Theory. It is a pleasure to thank the Institute and the organizers for excellent working conditions. We thank Paul Nelson for comments and providing us with an early version of~\cite{Nelson:stable} and Gergely Harcos and Dinakar Ramakrishnan for helpful comments. Also we thank the referee for a careful reading. The first author is supported by the Sloan fellowship BR2011-083 and the NSF grant DMS-1068043.


\bibliographystyle{plain}


\end{document}